\documentclass[pdflatex,sn-mathphys-num]{sn-jnl}% Math and Physical Sciences Numbered Reference Style 
\usepackage{graphicx}%
\usepackage{multirow}%
\usepackage{amsmath,amssymb,amsfonts}%
\usepackage{amsthm}%
\usepackage{mathrsfs}%
\usepackage[title]{appendix}%
\usepackage{xcolor}%
\usepackage{textcomp}%
\usepackage{manyfoot}%
\usepackage{booktabs}%
\usepackage{algorithm}%
\usepackage{algorithmicx}%
\usepackage{algpseudocode}%
\usepackage{listings}%
\newtheorem{theorem}{Theorem}[section]
\newtheorem{lemma}[theorem]{Lemma}
\newtheorem{proposition}[theorem]{Proposition}

\theoremstyle{definition}

\newtheorem{definition}[theorem]{Definition}

\newtheorem{example}[theorem]{Example}

\newtheorem{remark}[theorem]{Remark}
\numberwithin{equation}{section}

\allowdisplaybreaks
\begin{document}
	
	\title[Article Title]{A Kantorovich-type variant of Gr\"unwald
		Interpolation Operators}
	
	
	\author*[]{\fnm{P. C.} \sur{Vinaya}}\email{vinayapc01@gmail.com}

	\affil*[1]{\orgdiv{Department of Mathematics}, \orgname{Cochin University of Science and Technology}, \orgaddress{\street{Kalamassery}, \city{Cochin-22}, \postcode{682022}, \state{Kerala}, \country{India}}}
	
	
	\abstract{
		In this paper, we introduce a new sequence of operators based on the Gr\"unwald interpolation operators on Chebyshev nodes on the space $L^p[0,\pi]$. The operators we consider are integral variants of the Gr\"unwald interpolation operators, inspired from the classical Kantorovich operators. Unlike the original Gr\"unwald interpolation operators, our construction enables the derivation of convergence results not only on $C[0,\pi]$ but also in the space $L^p[0,\pi]$. First, we establish the uniform boundedness of this sequence on these spaces and subsequently prove the convergence of the operators. We obtain quantitative estimates using modulus of continuity and a suitable K-functional. Furthermore, we derive a point-wise estimate via the Hardy–Littlewood maximal operator.
		By invoking a Korovkin-type theorem, we extend the convergence results to several Banach function spaces on a nontrivial subspace. In particular, we establish these results for weighted Lebesgue spaces, Grand Lebesgue spaces, Morrey spaces, Orlicz spaces etc.
	}

	
	

	
	
	\maketitle
	\section{Introduction and Preliminaries}\label{intro}
	The classical Kantorovich operators were introduced as a natural modification of the classical Bernstein operators in order to extend approximation process from continuous functions to the class of integrable functions. Since the Bernstein operators are defined using the pointwise evaluation of the function, this becomes restrictive when dealing with functions that are not continuous or are only defined almost everywhere, such as those in $L^p$ spaces. To overcome this limitation, Kantorovich replaced point evaluations with local averages, defining the operators in terms of integrals over small subintervals. In $1941$, G. Gr\"unwald introduced a sequence operators using the classical Lagrange interpolation operators on Chebychev nodes and established their uniform convergence on the space $C[0,\pi]$. This work aims to extend this convergence result to the space $L^p[0,\pi]$. By a similar approach to that of Kantorovich, we develop integral variants of these operators and  extend them in the $L^p$ space and more generally to a wide class of  Banach function spaces. Furthermore, we investigate not only the convergence behavior of these operators but also derive quantitative estimates that describe their rate of convergence.
	
	First, we recall the the definition of the Lagrange interpolation operators.
	\begin{definition}
		Let $x_1, x_2, \ldots, x_n$ be a sequence of distinct nodes in $[-1,1]$, and suppose that $f \in C[-1,1]$. 
		The classical Lagrange interpolation operator produces a unique polynomial of degree at most $n-1$ that takes the values 
		$f(x_1), f(x_2), \ldots, f(x_n)$ at the corresponding points $x_1, x_2, \ldots, x_n$. This can be expressed as
		\[
		L_n(f)(x)=\sum\limits_{k=1}^nf(x_k)P_k(x),\ x\in [-1,1],
		\]
		where $P_k$ denotes the fundamental polynomials of the Lagrange interpolation operators given by 
		\[
		P_k(x)=\frac{\zeta(x)}{\zeta^\prime(x_k)(x-x_k)},
		\]
		$k=1,2,\ldots,n$ and the polynomials $\zeta(x)$ is defined by 
		$\zeta(x)=c(x-x_1)(x-x_2)\ldots (x-x_n),$
		$c$ is arbitrary non-zero constant.
	\end{definition}
	We also have  $\sum\limits_{k=1}^nP_k(x)=1$ for every $x\in[-1,1]$. In this work, we restrict our attention to the Chebyshev nodes of first kind, which are given by
	\[
	x_k = \cos \theta_k^{(n)}, \quad \text{where} \quad \theta_k^{(n)} = \frac{(2k - 1)\pi}{2n}, \quad k = 1, 2, \dots, n.
	\]
	
	The Lagrange interpolation operator based on Chebyshev nodes can be expressed as
	\[
	L_n(f)(\theta) = \sum_{k=1}^n f(\cos \theta_k^{(n)})(-1)^{k+1} \frac{\cos(n\theta) \sin \theta_k^{(n)}}{n(\cos\theta - \cos \theta_k^{(n)})},
	\]
	where
	\[
	P_k(\theta) = (-1)^{k+1} \frac{\cos(n\theta) \sin \theta_k^{(n)}}{n(\cos\theta - \cos \theta_k^{(n)})},
	\quad k = 1, 2, \dots, n.
	\]
	The problem of convergence of the Lagrange interpolation operators on Chebyshev nodes has been studied by many mathematicians. 
	In 1914, G.~Faber proved the existence of a continuous function for which $L_n(f)$ does not converge to $f$ at some point in $[-1,1]$ (see \cite{faber}). 
	Later, Gr\"unwald and Marcinkiewicz proved the existence of a function $f \in C[-1,1]$ such that, for all $x \in [-1,1]$, the sequence $L_n(f)(x)$ diverges. 
	This result is known as the Gr\"unwald--Marcinkiewicz theorem (see \cite{mills}).
	
	Despite this divergence result, in 1941, G.~Gr\"unwald established a convergence result for an averaged sequence of operators constructed using the Lagrange interpolation operators on Chebyshev nodes. 
	He obtained the following theorem.
	
	\begin{theorem}[\textbf{G. Grünwald}, \cite{grunwald}]\label{grun}
		Let $f\in C[-1,1]$, then 
		\[
		\lim\limits_{n\rightarrow\infty}\frac{1}{2}\{L_n(f)(\theta-\frac{\pi}{2n})+L_n(f)(\theta+\frac{\pi}{2n})\}=f(\cos(\theta)),
		\]
		where the convergence is uniform on the interval $[0,\pi]$.
	\end{theorem}
	The following lemma plays an important role in obtaining the theorem.
	\begin{lemma}[\textbf{G. Grünwald}, \cite{grunwald}]\label{lem}
		\[
		\frac{1}{2}\sum\limits_{k=1}^n \left| P_k\left(\theta - \frac{\pi}{2n}\right) + P_k\left(\theta + \frac{\pi}{2n}\right) \right| < c_1,
		\]
		where $c_1>0$ is an absolute constant
	\end{lemma} 
	The operators considered by Grünwald are as follows. 
	\begin{definition}\label{def}
		For $n=1,2,\ldots$, define $G_n:C[0,\pi]\rightarrow C[0,\pi]$ by 
		\[
		G_n(f)(\theta):=\frac{1}{2}\sum\limits_{k=1}^nf(\theta_k^{(n)})\{P_k(\theta-\frac{\pi}{2n})+P_k(\theta+\frac{\pi}{2n})\}.
		\]
	\end{definition}
	We observe that $G_n$ is also an interpolation operator on $C[0,\pi]$.
	We refer to these operators as Gr\"unwald interpolation operators, or simply Gr\"unwald operators.
	From Lemma~\ref{lem}, it follows that $\|G_n(f)\|_\infty \leq c_1 \|f\|_\infty$ for all $f \in C[0,\pi]$, where $c_1$ does not depend on $n$ or $f$.
	Note that if $f \in C[-1,1]$, then $f \circ \cos \in C[0,\pi]$, where $\circ$ denotes the operation composition. 
	Hence, by Theorem~\ref{grun}, we have
	\[
	\lim_{n \to \infty} \|G_n(f \circ \cos) - f \circ \cos\|_\infty = 0
	\quad \text{for all } f \in C[-1,1].
	\] 
	By a slight modification of Gr\"unwald's proof, we also have $\lim\limits_{n\to\infty}\|G_n(f)-f\|_\infty=0$ for all $f\in C[0,\pi]$. Our aim is to extend the definition of to Grünwald operators to the space $L^p[0,\pi]$ for $1\leq p<+\infty$. A natural extension is not possible, since Grünwald operators are not bounded in the $L^1$ norm. A proof of this observation is given below.
	\begin{theorem}
		The operators $G_n$, $n=1,2,\ldots$ are not bounded with respect to the $L^1$-norm on $[0,\pi]$.
	\end{theorem}
	\begin{proof}
		We will show that for any \(n,m \in \mathbb{N}\), there exists a function \(f_{n,m} \in C[0,\pi]\) such that 
		\[
		\|G_n(f_{n,m})\|_1 > m \|f_{n,m}\|_1.
		\]
		Let 
		\[
		C_n = \int_0^{\pi} \left| P_n\left(\theta - \frac{\pi}{2n}\right) + P_n\left(\theta + \frac{\pi}{2n}\right) \right| d\theta
		\]
		and let \(N \in \mathbb{N}\) be such that \(\frac{1}{2^N} < \frac{C_n}{m}\).
		
		Define \(f_{n,m} \in C[0,\pi]\) by 
		\begin{equation*}
			f_{n,m}(x) =
			\begin{cases}
				0 & \text{if } 0 \leq x \leq \frac{(2n-1)\pi}{2n} - \frac{1}{2^{N+1}}, \\
				m & \text{if } x = \frac{(2n-1)\pi}{2n}, \\
				\dfrac{m \left( 2^{N+1} x - 2^{N+1} \frac{(2n-1)\pi}{2n} + 1 \right)}{1 - 2^{N+1} \frac{(2n-1)\pi}{2n}} & \text{if } \frac{(2n-1)\pi}{2n} - \frac{1}{2^{N+1}} \leq x \leq \frac{(2n-1)\pi}{2n}, \\
				\dfrac{m \left( -2^{N+1} x + 2^{N+1} \frac{(2n-1)\pi}{2n} + 1 \right)}{1 + 2^{N+1} \frac{(2n-1)\pi}{2n}} & \text{if } \frac{(2n-1)\pi}{2n} \leq x \leq \frac{(2n-1)\pi}{2n} + \frac{1}{2^{N+1}}, \\
				0 & \text{if } \frac{(2n-1)\pi}{2n} + \frac{1}{2^{N+1}} \leq x \leq \pi.
			\end{cases}
		\end{equation*}
		
		Now we see that \(\|G_n(f_{n,m})\|_1 = \frac{m C_n}{2}\) and \(\|f_{n,m}\|_1 = \frac{m}{2^{N+1}}\). Also,
		\[
		\|G_n(f_{n,m})\|_1 = \frac{m C_n}{2} = \frac{m^2 C_n}{2 m} > \frac{m^2}{2^{N+1}} = m \|f_{n,m}\|_1.
		\]
		Thus, \(G_n\) is not continuous with respect to the \(L^1\) norm.
	\end{proof}
	Therefore, we seek modified Grünwald operators that can be extended to the space $L^p[0,\pi]$.
	
	In this article, we define a new sequence of operators using the Grünwald interpolation operators on $L^p[0,\pi]$ for $1\leq p\leq+\infty$. We call these operators as Gr\"unwald-Kantorovich operators. The operators considered here are integral variants of Gr\"unwald operators analogous to the classical Kantorovich operators. Such variants appear in the literature; for instance, for an integral variant of Lagrange-type interpolation, we may refer to \cite{campiti}. Kantorovich variants of various operators and their approximation properties have been studied in \cite{barbosu, icoz, kajla}.
	
	It is also worthwhile to note that the new sequence of operators is non-positive. This follows directly from their very construction and the non-positivity of the Lagrange interpolation operators.
	
	We observe that, like Grünwald operators, the new sequence of operators exhibits similar convergence properties not only on the spaces $C[0,\pi]$ but also in $L^p[0,\pi]$. To establish these convergence results, we first prove the boundedness of the operators and then derive the corresponding convergence results. We also obtain the rate of convergence of the Gr\"unwald-Kantorovich operators using the modulus of continuity on $C[0,\pi]$ and K-functionals in the case of $L^p[0,\pi]$. Moreover, we prove that the Gr\"unwald-Kantorovich operators are point-wise bounded by the Hardy-Littlewood maximal operator. We then extend the convergence results to certain classes of Banach function spaces on a nontrivial subspace. For this purpose, we employ non-positive Korovkin-type theorem in the setting of Banach function spaces together with the boundedness of the Hardy–Littlewood maximal operator. In particular, we obtain convergence results in various Banach function spaces, including weighted Lebesgue spaces, Grand Lebesgue spaces, weighted Grand Lebesgue spaces, Orlicz spaces, Lebesgue spaces with variable exponent and Morrey spaces.
	
	This article is organized as follows. In the next section, we introduce a new sequence of operators on $L^p[0,\pi]$. 
	We show that these operators are bounded in $C[0,\pi]$ and establish uniform convergence properties similar to those of Grünwald operators. 
	We also discuss their rate of convergence and determine the order of convergence using the modulus of continuity. 
	In the subsequent section, we prove that the sequence of operators are bounded in $L^p[0,\pi]$ and establish similar convergence results in this space. We further obtain the rate of convergence using a suitable K-functional. In the following section, we establish the point-wise boundedness of the sequence via the Hardy-Littlewood maximal function. Next, we recall Korovkin-type results obtained in the setting of Banach function spaces. Finally, we extend the convergence results to specific Banach function spaces mentioned above.
	\section{Grünwald-Kantorovich Operators}\label{section1}
	In this section, we introduce a new sequence of operators on $L^p[0,\pi]$ for $1\leq p\leq +\infty$. These operators are inspired by the classical Kantorovich operators, which are, in fact, integral variants of the classical Berntein operators \cite{kantorovich}. In a similar manner, we modify the Gr\"unwald interpolation operators to obtain their integral variants.
	
	In the present section, we restrict our attention to the space $C[0,\pi]$. We establish several properties of these operators, including boundedness, convergence and rates of convergence.
	
	The new sequence of operators is defined as follows.
	\begin{definition}
		Let $1\leq p\leq +\infty$, then $\mathcal{GK}_n:L^p[0,\pi]\to L^p[0,\pi]$ be defined as follows.
		\begin{equation}\label{gk_n}
			\mathcal{GK}_n(f)(\theta)=\frac{n}{\pi}\sum\limits_{k=1}^n \int\limits_{\theta_k^{(n)}}^{\theta_k^{(n)}+\frac{\pi}{2n}}f(t)\ dt \{P_k(\theta+\frac{\pi}{2n})+P_k(\theta-\frac{\pi}{2n})\}.
		\end{equation}
	\end{definition}
	We refer to these operators as Grünwald-Kantorovich operators.
	\begin{remark} 
	We remark that, $\mathcal{GK}_n(1)=1$ for all $n\in \mathbb{N}$, since $\frac{1}{2}\sum\limits_{k=1}^n P_k(\theta+\frac{\pi}{2n})+P_k(\theta-\frac{\pi}{2n})=1$ for all $\theta\in [0,\pi]$. We denote $S_{k,n}(\theta):=\frac{1}{2}\{P_k(\theta+\frac{\pi}{2n})+P_k(\theta-\frac{\pi}{2n})\}$. Therefore, from Lemma \ref{lem}, we have $\sum\limits_{k=1}^n |S_{k,n}(\theta)|\leq c_1$ for all $\theta\in C[0,\pi]$. 
	\end{remark}
	Now we prove the boundedness of the sequence $\{\mathcal{GK}_n\}_{n\in\mathbb{N}}$ in $C[0,\pi]$.
	\begin{lemma}\label{bdd}
		The sequence of operators $\{\mathcal{GK}_n:C[0,\pi]\to C[0,\pi]\}_{n\in\mathbb{N}}$ is bounded uniformly in the space $C[0,\pi]$.
	\end{lemma}
	\begin{proof}
		Let $\theta\in [0,\pi]$. We have 
		\begin{align*}
			|\mathcal{GK}_n(f)(\theta)|&\leq\frac{ 2n}{\pi}\sum\limits_{k=1}^n \int\limits_{\theta_k^{(n)}}^{\theta_k^{(n)}+\frac{\pi}{2n}}|f(t)|\ dt |S_{k,n}(\theta)|\\
			&\leq \|f\|_\infty \frac{2n}{\pi}\frac{\pi}{ 2n}\sum\limits_{k=1}^n|S_{k,n}(\theta)|\\
			&=\|f\|_\infty \sum\limits_{k=1}^n|S_{k,n}(\theta)|\leq c_1\|f\|_\infty.
		\end{align*}
		Taking supreme over $[0,\pi]$, we obtain $\|\mathcal{GK}_n(f)\|_\infty\leq c_1\|f\|_\infty$. As already seen, $c_1$ is independent of $n$ and $f$.
	\end{proof}
	In what follows, we establish the convergence and prove a convergence rate for the Grünwald-Kantorovich operators in the space $C[0,\pi]$.
	\begin{theorem}\label{main}
		For all $f\in C[0,\pi]$, we have 
		\[
		\lim\limits_{n\to\infty}\|\mathcal{GK}_n(f)-f\|_\infty= 0,
		\]
		where the sup-norm $\|.\|_\infty$ is taken over $[0,\pi]$.
	\end{theorem}
	\begin{proof}
		Let $f\in C[0,\pi]$ and $\theta\in [0,\pi]$. Then $f(\theta)$ can be expressed as 
		\[
		f(\theta)=\frac{2n}{\pi}\sum\limits_{k=1}^n \int\limits_{\theta_k^{(n)}}^{\theta_k^{(n)}+\frac{\pi}{ 2n}}f(\theta)\ dt S_{k,n}(\theta).
		\]
		By a change of variable, $\int\limits_{\theta_k^{(n)}}^{\theta_k^{(n)}+\frac{\pi}{2n}}f( t)\ dt=\int\limits_{0}^{\frac{\pi}{2n}}f(t+\theta_k^{(n)})\ dt$, from which we have 
		\[
		\mathcal{GK}_n(f)(\theta)-f(\theta)=\frac{ 2n}{\pi}\int\limits_{0}^{\frac{\pi}{2n}}\sum\limits_{k=1}^n (f(t)-f(\theta)) S_{k,n}(\theta)\ dt.
		\]
		Thus, we have
		\begin{align*}
			|\mathcal{GK}_n(f)(\theta)-f(\theta)|
			&=\frac{2n}{\pi}\sum\limits_{k=1}^n \int\limits_{\theta_k^{(n)}}^{\theta_k^{(n)}+\frac{\pi}{2n}}|f(t)-f(\theta)|\ dt |S_{k,n}(\theta)|\\
			&=\frac{2n}{\pi}\int\limits_{0}^{\frac{\pi}{ 2n}}\sum\limits_{k=1}^n |f(t+\theta_k^{(n)})-f(\theta)| |S_{k,n}(\theta)|\ dt\\
			&\leq \frac{2n}{\pi}\int\limits_{0}^{\frac{\pi}{ 2n}}\sum\limits_{k=1}^n |f(t+\theta_k^{(n)})-f( \theta_k^{(n)})| |S_{k,n}(\theta)|\ dt+\\
			&\qquad\qquad \frac{2n}{\pi}\int\limits_{0}^{\frac{\pi}{ 2n}}\sum\limits_{k=1}^n |f(\theta_k^{(n)})-f(\theta)| |S_{k,n}(\theta)|\ dt	\\
			& = A_1+A_2.
		\end{align*}
		Now, we consider $A_1$. Since $f$ is uniformly continuous on $[0,\pi]$, for a given $\epsilon>0$, for sufficiently large $n$, whenever $|t|\leq \frac{\pi}{n}$, we have 
		\[
		|f(t+\theta_k^{(n)})-f(\theta_k^{(n)})|<\frac{\epsilon}{2c_1}, 
		\] 
		for all $k=1,2,\ldots, n$. Therefore , we have $A_1\leq \frac{\epsilon}{2}$. Also, from the proof of Theorem \ref{grun} in \cite{grunwald}, for sufficiently large $n$, we have
		\[ 
		\sup\limits_{\theta\in [0,\pi]}\sum\limits_{k=1}^n |f(\theta_k^{(n)})-f(\theta)| |S_{k,n}(\theta)|\leq \frac{\epsilon}{2}.
		\]
		Therefore, $A_2\leq \frac{\epsilon}{2}$.
		Thus, we obtain $\|\mathcal{GK}_n(f)-f\|_\infty\leq \epsilon.$ 
		Hence, the result.
	\end{proof}
	\subsection{Rate of convergence}
	In this section, we determine the order of convergence for the Grünwald- Kantorovich operators. To this end, we employ the modulus of continuity. The modulus of continuity of a bounded function $f:I \to \mathbb{R}$ on a real interval $I$, with argument $\delta > 0$, is defined by
	\[
	\omega(f,\delta) = \sup\{\,|f(x)-f(y)| : |x-y| \leq \delta,\; x,y \in I\,\},
	\]
	which may be denoted simply by $\omega(\delta)$. 
	
	We now proceed to prove the result.
	\begin{theorem}
		Let $f\in C[0,\pi]$ and $\theta\in [0,\pi]$, then 
		\[
		|\mathcal{GK}_n(f)(\theta)-f( \theta)|=\mathcal{O}(\omega(f\circ \arccos,\frac{\sqrt{1-x^2}}{n})+\omega(f\circ \arccos, \frac{1}{n^2})+\omega(f,\frac{\pi}{n})),
		\]
		where $x=\cos\theta$.
	\end{theorem}
	\begin{proof}
		We have	
		\begin{align*}
			\mathcal{GK}_n(f)(\theta)-f(\theta)&=\frac{2n}{\pi}\int\limits_{0}^{\frac{\pi}{ 2n}}\sum\limits_{k=1}^n f(t+\theta_k^{(n)})-f( \theta_k^{(n)}) S_{k,n}(\theta)\ dt+\\ &\qquad \frac{2n}{\pi}\int\limits_{0}^{\frac{\pi}{2n}}\sum\limits_{k=1}^n f(\theta_k^{(n)})-f(\theta) S_{k,n}(\theta)\ dt\\
			&=E_1+E_2.
		\end{align*} 
		$E_2$ can be expressed as,
		$
		E_2=\frac{2n}{\pi}\int\limits_{0}^{\frac{\pi}{ 2n}}(G_n(f)(\theta)-f(\theta))\ dt. 
		$
		From \cite{mills1}, we have $|G_n(f)(\theta)-f(\theta)|=\mathcal{O}(\omega(f\circ \arccos,\frac{\sqrt{1-x^2}}{n})+\omega(f\circ \arccos, \frac{1}{n^2}))$, where $x=\cos\theta$. Thus, $|E_2|=\mathcal{O}(\omega(f\circ\arccos, \frac{\sqrt{1-x^2}}{n})+\omega(f\circ\arccos, \frac{1}{n^2}))$. Now, consider $E_1$. By the properties of the modulus of continuity, we have 
		\[
		|f(t+\theta_k^{(n)})-f(\theta_k^{(n)})|\leq \omega(f, |t|)\leq \omega(f, \frac{\pi}{n}),
		\]
		for $k=1,2,\ldots, n$. Using Lemma \ref{lem}, we have  $|E_2|=\mathcal{O}(\omega(f,\frac{\pi}{n})).$
	\end{proof} 
	In the following section, we discuss the extension of the  Grünwald-Kantorovich operators to the space $L^p[0,\pi]$ for $1\leq p<+\infty$ and establish their convergence properties.
	\section{Convergence in the $L^p$-space}
	In this section, we prove the uniform boundedness property of the Gr\"unwald Kantorovich operators on the space $L^p[0.\pi]$ for $1\leq p<+\infty$.
	
	First, we establish the uniform boundedness of $\{\mathcal{GK}_n\}_{n\in\mathbb{N}}$ in the space $L^1[0,\pi]$. Then, we extend this property to the space $L^p[0,\pi]$. Before proving this, we have the following proposition which is crucial in obtaining our main results.
	\begin{proposition}\label{propos}
		For all $k=1,2,\ldots,n$, we have
		\begin{enumerate}
			\item[(i)] \[
			\left (\int\limits_{0}^{\pi} |\theta-\theta_k^{(n)}|^p \, |S_{k,n}(\theta)|^p \, d\theta\right )^\frac{1}{p}
			=
			\begin{cases}
				\mathcal{O}\!\left(\dfrac{\log n + 1}{n^2}\right), & p = 1, \\[8pt]
				\mathcal{O}\!\left(\dfrac{1}{n^{1+\frac{1}{p}}}\right), & 1<p<+\infty.
			\end{cases}
			\]
			\item[(ii)]
			\[
			\left (\int\limits_{0}^{\pi} |S_{k,n}(\theta)|^p \, d\theta\right )^\frac{1}{p}
			=\mathcal{O}\left (\frac{1}{n^\frac{1}{p}}\right ), \quad 1\leq p<+\infty.
\			\]
		
		\end{enumerate}
	\end{proposition}
	\begin{proof}
		First, we prove $(i)$.
		For $\theta\neq \theta_k^{(n)}+\frac{\pi}{2n}$, $k=1,2,\ldots, n$ we have the the following inequlity from \cite{grunwald},
		\[
		|S_{k,n}(\theta)|\leq \frac{\pi^3}{4n^2}\frac{1}{(\theta-\theta_k^{(n)}-\frac{\pi}{2n})^2}.
		\]
		We split the interval $[0,\pi]$ as into two, that is, $|\theta-\theta_k^{(n)}|>\frac{\pi}{n}$ and $|\theta-\theta_k^{(n)}|\leq \frac{\pi}{n}$.
		We have 
		\begin{align*}
			\int\limits_{|\theta-\theta_k^{(n)}|>\frac{\pi}{n}} \frac{|\theta_k^{(n)}-\theta|} {{(\theta-\theta_k^{(n)}-\frac{\pi}{2n})^2}}\ d\theta 
			&\leq	4\int\limits_{|\theta-\theta_k^{(n)}|>\frac{\pi}{n}}\frac{1}{|\theta-\theta_k^{(n)}|}\ d\theta=\mathcal{O}(\log n).
		\end{align*}
		Therefore, we have 
		\[
		\int\limits_{|\theta-\theta_k^{(n)}|>\frac{\pi}{n}} |\theta-\theta_k^{(n)}| \, |S_{k,n}(\theta)|\, d\theta\leq \frac{\pi^3}{4n^2}\mathcal{O}(\log n)=\mathcal{O}(\frac{\log n}{n^2}).
		\]
		Since $|S_k,n(\theta)|\leq c_1$ for all $\theta\in [0,1]$, we have 
		\[
		\int\limits_{|\theta-\theta_k^{(n)}|\leq \frac{\pi}{n}} |\theta_k^{(n)}-\theta||S_{k,n}(\theta)|\ d\theta \leq \mathcal{O}(\frac{1}{n^2}).
		\]
		Therefore, we have 
		\[
		\int\limits_{0}^{\pi} |\theta-\theta_k^{(n)}| \, |S_{k,n}(\theta)|\, d\theta=\mathcal{O}\left (\frac{1+\log n}{n^2}\right ).
		\]
		Hence, the first part is proved. 
		To prove the second part, let $p>1$ and consider the integral 
		\begin{align*}
			\int\limits_{|\theta-\theta_k^{(n)}|>\frac{\pi}{n}}|\theta-\theta_k^{(n)}|^p|S_{k,n}(\theta)|^p\ d\theta\leq \frac{\pi^3}{4n^2}\int\limits_{|\theta-\theta_k^{(n)}|>\frac{\pi}{n}}\frac{|\theta-\theta_k^{(n)}|^p}{|\theta-\theta_k^{(n)}-\frac{\pi}{2n}|^{2p}}\ d\theta.
		\end{align*}
		By using binomial expansion, for $p>1$, we have 
		\[
		|\theta-\theta_k^{(n)}|^p
		\leq 
		\sum_{j=0}^{\infty}
		\frac{|p(p-1)\cdots(p-j+1)|}{j!}
		\left(\frac{\pi}{2n}\right)^j
		\left |\theta-\theta_k^{(n)}-\frac{\pi}{2n}\right |^{p-j}.
		\] 
		Therefore, we have 
		\[
		\frac{|\theta-\theta_k^{(n)}|^p}{|\theta-\theta_k^{(n)}-\frac{\pi}{2n}|^{2p}}
		\leq 
		\sum_{j=0}^{\infty}
		\frac{|p(p-1)\cdots(p-j+1)|}{j!}
		\left(\frac{\pi}{2n}\right)^j
		\left |\theta-\theta_k^{(n)}-\frac{\pi}{2n}\right |^{-p-j}.
		\] 
		Then, for all $j=0,1,\ldots$, we have 
		\begin{align*}
			\left (\frac{\pi}{2n}\right )^{j} \int\limits_{|\theta-\theta_k^{(n)}|>\frac{\pi}{n}}\left|\theta-\theta_k^{(n)}-\frac{\pi}{2n}\right |^{-p-j}\ d\theta&\leq \frac{2}{p+j-1}\left (\frac{\pi}{2n}\right )^{j}\left (\frac{\pi}{2n}\right )^{1-(p+j)}\\
			&=\frac{2}{p+j-1}\left (\frac{\pi}{2n}\right )^{1-p}.
		\end{align*}
		Let \(p>1\) and \(j\in\mathbb{N}\). Then, we get
		\[
		\left|p(p-1)\cdots(p-j+1)\right|
		= \prod_{m=0}^{j-1} |p-m|.
		\]
		Let \(N=\lceil p \rceil\). Splitting the product, we obtain
		\[
		\prod_{m=0}^{j-1} |p-m|
		= \left(\prod_{m=0}^{N} |p-m|\right)
		\left(\prod_{m=N+1}^{j-1} |p-m|\right).
		\]
		For \(m>N\), we have \(|p-m|=m-p\le m\), hence
		\[
		\prod_{m=N+1}^{j-1} |p-m|
		\le \prod_{m=N+1}^{j-1} m
		= \frac{(j-1)!}{N!}.
		\]
		Thus, we get
		\[
		\left|p(p-1)\cdots(p-j+1)\right|
		\le C_p (j-1)!,
		\]
		for some constant \(C_p>0\).
		Now, we have 
		\begin{align*}
			\int\limits_{|\theta-\theta_k^{(n)}|>\frac{\pi}{n}}\frac{|\theta-\theta_k^{(n)}|^p}{|\theta-\theta_k^{(n)}-\frac{\pi}{2n}|^{2p}}\ d\theta
			&\leq 
			\left (\frac{\pi}{2n}\right )^{1-p}\sum_{j=0}^{\infty}
			\frac{|p(p-1)\cdots(p-j+1)|}{(p+j-1)j!}\\
			&\leq C_p\left (\frac{\pi}{2n}\right )^{1-p}\sum_{j=0}^{\infty}
			\frac{1}{(p+j-1)j}\\
			&\leq \mathcal{C}_p \left (\frac{\pi}{2n}\right )^{1-p},
		\end{align*}
		where $\mathcal{C}_p>0$ is a constant since the series is convergent. Therefore, we obtain 
		\[
		\frac{\pi^3}{4n^2}\left (\int\limits_{|\theta-\theta_k^{(n)}|>\frac{\pi}{n}}\frac{|\theta-\theta_k^{(n)}|^p}{|\theta-\theta_k^{(n)}-\frac{\pi}{2n}|^{2p}}\ d\theta\right )^\frac{1}{p}=\mathcal{O}\left(\frac{1}{n^{1+\frac{1}{p}}}\right )
		\]
		For $|\theta-\theta_k^{(n)}|\leq \frac{\pi}{n}$, we have 
		\[
		\left (\int\limits_{|\theta-\theta_k^{(n)}|\leq \frac{\pi}{n}}|\theta-\theta_k^{(n)}|^p||S_{k,n}(\theta)|^p\ d\theta\right )^\frac{1}{p}=\mathcal{O}\left (\frac{1}{n^{1+\frac{1}{p}}}\right )
		\]
		Now we prove $(ii)$. To do this, we split $[0,\pi]$ into three intervals, that is, $|\theta-\theta_k^{(n)}|\leq \frac{\pi}{n}$, $\frac{\pi}{n}<|\theta-\theta_k^{(n)}|\leq  \frac{\pi}{n^{\frac{1}{3}}}$ and $|\theta-\theta_k^{(n)}|>\frac{\pi}{n^{\frac{1}{3}}}$.
		
		\textbf{Case 1: $|\theta-\theta_k^{(n)}|\leq \frac{\pi}{n}$} In this case, we have 
		\[
		\left (\int\limits_{|\theta-\theta_k^{(n)}|\leq \frac{\pi}{n}} |S_{k,n}(\theta)|^p \, d\theta\right )^\frac{1}{p}\leq 2c_1\frac{\pi}{n^\frac{1}{p}}=\mathcal{O}\left (\frac{1}{n^\frac{1}{p}}\right ).
		\]
		\textbf{Case 2: $|\theta-\theta_k^{(n)}|>\frac{\pi}{n^{\frac{1}{3}}}$}, we have 
		\begin{align*}
			|S_{k,n}(\theta)|
			&\le \frac{\pi^3}{4n^2}\frac{1}{\left(\theta-\theta_k^{(n)}-\frac{\pi}{2n}\right)^2} \\
			&\le \frac{\pi^3}{4n^2}\frac{1}{\left(|\theta-\theta_k^{(n)}|-\frac{\pi}{2n}\right)^2} \\
			&\le \frac{\pi^3}{4n^2}\frac{1}{\left(\frac{\pi}{n^{1/3}}-\frac{\pi}{2n}\right)^2} \\
			&= \frac{\pi^3}{4n^2}\cdot \frac{1}{\frac{\pi^2}{n^{2/3}}\left(1-\frac{1}{2n^{1/3}}\right)^2} \leq \frac{\pi}{2n^{4/3}}.
		\end{align*}
		Thus, we get 
		\[
		\left (\int\limits_{|\theta-\theta_k^{(n)}|>\frac{\pi}{n^\frac{1}{3}}} |S_{k,n}(\theta)|^p \, d\theta\right )^\frac{1}{p}=\mathcal{O}\left (\frac{1}{n^{\frac{4}{3}}}\right ).
		\]
		\textbf{Case 3: $\frac{\pi}{n}<|\theta-\theta_k^{(n)}|\leq  \frac{\pi}{n^{\frac{1}{3}}}$}
		Let
		\[
		I=\int\limits_{\frac{\pi}{n}<|\theta-\theta_k^{(n)}|\le \frac{\pi}{n^{1/3}}}
		\frac{d\theta}{\left|\theta-\theta_k^{(n)}-\frac{\pi}{2n}\right|^{2p}}.
		\]
		The given interval does not contain any singularities of the integrand. Therefore, on integrating, we have
		\begin{align*}
			I&=\int\limits_{\theta_k^{(n)}+\frac{\pi}{n}}^{\theta_k^{(n)}+\frac{\pi}{n^{1/3}}}
			\frac{d\theta}{\left(\theta-\theta_k^{(n)}-\frac{\pi}{2n}\right)^{2p}}
			+
			\int\limits_{\theta_k^{(n)}-\frac{\pi}{n^{1/3}}}^{\theta_k^{(n)}-\frac{\pi}{n}}
			\frac{d\theta}{\left(\theta_k^{(n)}+\frac{\pi}{2n}-\theta\right)^{2p}}\\
			&=\frac{1}{2p-1}
			\left[
			\left(\frac{\pi}{2n}\right)^{1-2p}
			-
			\left(\frac{\pi}{n^{1/3}}-\frac{\pi}{2n}\right)^{1-2p}
			\right]\\
			&\quad + \frac{1}{2p-1}
			\left[
			\left(\frac{3\pi}{2n}\right)^{1-2p}
			-
			\left(\frac{\pi}{n^{1/3}}+\frac{\pi}{2n}\right)^{1-2p}
			\right]\\
			&=\frac{1}{2p-1}
			\Bigg[
			\left(\frac{\pi}{2n}\right)^{1-2p}
			+
			\left(\frac{3\pi}{2n}\right)^{1-2p}
			-
			\left(\frac{\pi}{n^{1/3}}-\frac{\pi}{2n}\right)^{1-2p}
			-
			\left(\frac{\pi}{n^{1/3}}+\frac{\pi}{2n}\right)^{1-2p}
			\Bigg]\\
			&=\mathcal{O}\big(n^{2p-1}\big).
		\end{align*}
		Thus, we have  $\frac{\pi^3}{4n^2}I^\frac{1}{p}=\mathcal{O}\left (\frac{1} {n^\frac{1}{p}}\right )$. Hence the proposition.
	\end{proof}
	Now, we prove our main result.
	\begin{lemma}\label{l1bdd}
		The sequence of operators $\{\mathcal{GK}_n:L^1[0,\pi]\to L^1[0,\pi]\}_{n\in\mathbb{N}}$ is bounded uniformly in the space $L^1[0,\pi]$.
	\end{lemma}
	\begin{proof}
		By Proposition \ref{propos}, we have $\int\limits_0^\pi|S_{k,n}(\theta)|\ d\theta=\mathcal{O}(\frac{1}{n})$.
		Hence, we have 
		\begin{align*}
			\int\limits_{0}^{\pi}	|\mathcal{GK}_n(f)(\theta)|\ d\theta &\leq \frac{2n}{\pi}\sum\limits_{k=1}^n \int\limits_{\theta_k^{(n)}}^{\theta_k^{(n)}+\frac{\pi}{2n}}|f(t)|\ dt\int\limits_{0}^{\pi} |S_{k,n}(\theta)|\ d\theta\\
			&\leq \frac{2Mn}{n\pi}\sum\limits_{k=1}^n \int\limits_{\theta_k^{(n)}}^{\theta_k^{(n)}+\frac{\pi}{2n}}|f(t)|\ dt\\
			&=\frac{M}{\pi}\sum\limits_{k=1}^n \int\limits_{\theta_k^{(n)}}^{\theta_k^{(n)}+\frac{\pi}{2n}}|f(t)|\ dt,
		\end{align*}
		where $\int\limits_{0}^{\pi} |S_{k,n}(\theta)|\ d\theta\leq \frac{M}{n}$ some constant $M>0$. Thus, we conclude that $\|\mathcal{GK}_n(f)\|_1\leq C\|f\|_1$, for some constant $C>0$. We also note that $C$ does not depend on $n$ or $f$.
	\end{proof}
	\subsection{Boundedness in the $L^p$-space}
	In the preceding section, we established the uniform boundedness of the sequence of operators $\{\mathcal{GK}_n\}_{n\in\mathbb{N}}$ in the spaces $C[0,\pi]$ (or $L^\infty[0,\pi]$) and $L^1[0,\pi]$. In this section, we prove that this uniform boundedness can be established in the space $L^p[0,\pi]$, for $1<p<+\infty$. To prove this, we recall a fundamental interpolation theorem in harmonic analysis, namely the Riesz--Thorin interpolation theorem (see Theorem 1.19, \cite{fourier}), stated below.
	\begin{theorem}{Riesz--Thorin Interpolation}\label{RT}
		Let $1\leq p_0,p_1,q_0,q_1\leq +\infty$, and for $0<\theta<1$, define $p$ and $q$ by 
		\begin{equation}\label{con}
			\frac{1}{p}=\frac{1-\theta}{p_0}+\frac{\theta}{p_1},\quad 	\frac{1}{q}=\frac{1-\theta}{q_0}+\frac{\theta}{q_1}.
		\end{equation}
		If $T$ is a linear operator from $L^{p_0}+L^{p_1}
		$ to $L^{q_0}+L^{q_1}$ such that 
		\[
		\|Tf\|_{q_0}\leq M_0\|f\|_{p_0}, \ for\ f\in L^{p_0}
		\]
		and 
		\[
		\|Tf\|_{q_1}\leq M_1\|f\|_{p_1}, \ for\ f\in L^{p_1}.
		\]
		Then 
		\[
		\|Tf\|_{q}\leq M_0^{1-\theta}M_1^\theta\|f\|_{p}, \ for\ f\in L^{p}
		\]
	\end{theorem}
	Now, we have the following result. 
	\begin{lemma}\label{lpbdd}
		The sequence of operators $\{\mathcal{GK}_n:L^p[0,\pi]\to L^p[0,\pi]\}_{n\in\mathbb{N}}$ is bounded uniformly in the space $L^p[0,\pi]$, for $1<p<+\infty$.
	\end{lemma}
	\begin{proof}
		From Lemma \ref{bdd} and \ref{l1bdd}, we have  
		\[
		\|\mathcal{GK}_n(f)\|_\infty\leq c_1\|f\|_\infty, \ for \ f\in L^\infty[0,\pi] \ for\ all\ n\in\mathbb{N}
		\]
		and 
		\[
		\|\mathcal{GK}_n(f)\|_1\leq C\|f\|_1, \ for \ f\in L^1[0,\pi]\ for\ all\ n\in\mathbb{N}.
		\]
		
		Therefore, by applying Theorem \ref{RT}, we have 
		\[
		\|\mathcal{GK}_n(f)\|_p\leq c_1^{1-\theta}C^\theta\|f\|_p, \ for \ f\in L^p[0,\pi]\ for\ all\ n\in\mathbb{N}.
		\]
		for $\theta\in (0,1)$, $p_0=q_0=1$, $p_1=q_1=\infty$ and $p$, $q$ satisfying \ref{con}. In other words, we see that for $1<p<+\infty$, 
		\[
		\|\mathcal{GK}_n(f)\|_p\leq C_p\|f\|_p, \ for \ f\in L^p[0,\pi]\ for\ all\ n\in\mathbb{N},
		\]
		where $C_p$ does not depend on $n$ or $f$.
	\end{proof}
	\subsection{Convergence in the space $L^p[0,\pi]$}
	In this section, we establish the convergence of the sequence $\{\mathcal{GK}_n\}_{n\in\mathbb{N}}$ in the space $L^p[0,\pi]$ for $1\leq p<+\infty$. We have the following main result.
	\begin{theorem}\label{main2}
		Let $1\leq p<+\infty$. Then, we have 
		\[
		\lim\limits_{n\to\infty}\|\mathcal{GK}_n(f)-f\|_p=0,
		\]
		for all $f\in L^p[0,\pi]$.
	\end{theorem}
	\begin{proof}
		Let $g\in C[0,\pi]$.
		We have,
		\begin{equation}\label{eq}
			\|\mathcal{GK}_n(g)-g\|_p=(\int\limits_0^\pi|\mathcal{GK}_n(g)(\theta)-g(\theta)|^p\ d\theta)^\frac{1}{p}\leq \pi^\frac{1}{p}\|\mathcal{GK}_n(g)-g\|_\infty.
		\end{equation}
		By Theorem \ref{main}, $\lim\limits_{n\to\infty}\|\mathcal{GK}_n(g)-g\|_\infty=0$. Suppose $\epsilon>0$. We know that $C[0,\pi]$ is dense in $L^p[0,\pi]$. Thus, if $f\in L^p[0,\pi]$, then there exist a $g\in C[0,\pi]$ such that $\|f-g\|_p<\epsilon$. Let $f\in L^p[0,\pi]$, then 
		\begin{align*}
			\|\mathcal{GK}_n(f)-f\|_p&\leq \|\mathcal{GK}_n(f-g)\|_p+\|\mathcal{GK}_n(g)-g\|_p+
			\|g-f\|_p\\
			&\leq (C_p+1)\|f-g\|_p+ \|\mathcal{GK}_n(g)-g\|_p,
		\end{align*}
		since, $\|\mathcal{GK}_n(f)\|_p\leq C_p\|f\|_p$ for all $f\in L^p[0,\pi]$ by Lemma \ref{l1bdd} and Lemma \ref{lpbdd}. Hence by \ref{eq}, for sufficiently large $n$ and a given $\epsilon>0$,
		\[
		\|\mathcal{GK}_n(f)-f\|_p\leq (C_p+1)\epsilon+ \epsilon=(C_p+2)\epsilon.
		\]
		for some $C_p>0$ which does not depend on $n$, $f$ or $g$. Hence the proof.
	\end{proof}
	\begin{remark}
		The convergence established in Theorem \ref{main2} may alternatively be deduced by applying a Korovkin-type theorem from \cite{vinaya}. This result will also be employed to establish convergence results in certain Banach function spaces in Section~\ref{section5} 
	\end{remark}
	\subsection{Rate of Convergence}		
	In this section, our aim is to obtain a quantitative estimate for the convergence of the Gr\"unwald-Kantorovich operators on the space $L^p[0,\pi]$ for $1\leq p<+\infty$. To this end, we employ a suitable $K$-functional. We consider the following version of Petree's K-functional, in which the derivative is measured in the uniform norm in contrast to the K-functional where derivative is measured in the $L^p$ norm.
	
	For $1\leq p<+\infty$, $f \in L^p[0.\pi]$ and $\delta > 0$, we consider
	\[
	K_p(f,\delta)
	=
	\inf_{g\in C^1[0,\pi]}
	\left\{
	\|f-g\|_{p}
	+
	\delta \|g'\|_{\infty}
	\right\}.
	\]
	Here $C^1[0,\pi]$ denotes the space of all continuously differentiable functions on $[0,\pi]$.
	Applications of similar $K$-functionals for certain neural network-type operators can be found in~\cite{costarelli}. 
	
	We now state the following estimation theorem, which yields the rate of convergence in terms of the above-mentioned K-functional.
	\begin{theorem}\label{main3}
		Let $1\leq p<+\infty$ and $f\in L^p[0,\pi]$, then
	\[ 
	\|\mathcal{D}_n(f)-f\|_p\leq R_pK_p(f,\tilde R_pm_{n}),
	\]
	for some positive constants $R_p$ and $\tilde R_p$ and
		 \[
		m_{n}=\begin{cases}
			\frac{1+\log n}{n}, & p = 1, \\[8pt]
			\frac{1}{n}+\frac{1}{n^\frac{1}{p}}, & 1<p<+\infty.
		\end{cases}
		\]
	\end{theorem}
	\begin{proof}
		Let $f\in L^p[0,\pi]$, $g\in C^1[0,\pi]$ and $\theta\in [0,\pi]$. Then, 
		\[
		\mathcal{GK}_n(f)(\theta)-f(\theta)=\mathcal{GK}_n(f-g)(\theta)+\mathcal{GK}_n(g)(\theta)-g(\theta)+g(\theta)-f(\theta).
		\]
		Therefore, 
		\begin{align*}
			\|\mathcal{GK}_n(f)-f\|_p&\leq \|\mathcal{GK}_n(f-g)\|_p+\|\mathcal{GK}_n(g)-g\|_p+\|g-f\|_p.\\
			&\leq
			C_p\|f-g\|_p+\|\mathcal{GK}_n(g)-g\|_p+\|g-f\|_p\\ 
			&\leq (C_p+1)\|f-g\|_p+\|\mathcal{GK}_n(g)-g\|_p\\
			&=(C_p+1)\|f-g\|_p+\|\mathcal{GK}_n(g)-g\|_p, 
		\end{align*}
		since $\|\mathcal{GK}_n(f)\|_p\leq C_p\|f\|_p$ for some constant $C_p>0$.
		
		Using mean value theorem, for $t,\theta\in [0,\pi]$, there exists a $\xi$ between $t$ and $\theta$ such that $|g(t)-g(\theta)|=|g'(\xi)||t-\theta|\leq \|g'\|_\infty|t-\theta|$.
		Consider
		\begin{align*}
		|\mathcal{GK}_n(g)(\theta)-&g(\theta)|\\
		&= \frac{2n}{\pi} \sum_{k=1}^n \int\limits_{\theta_k^{(n)}}^{\theta_k^{(n)}+\frac{\pi}{2n}} |g(t)-g(\theta)| \, dt \, |S_{k,n}(\theta)|\\
		&\leq \frac{2n}{\pi}\|g'\|_\infty \sum_{k=1}^n\int\limits_{\theta_k^{(n)}}^{\theta_k^{(n)}+\frac{\pi}{2n}} |t-\theta| \, dt \, |S_{k,n}(\theta)|\\
		&\leq \frac{2n}{\pi}\|g'\|_\infty \sum_{k=1}^n\int\limits_{\theta_k^{(n)}}^{\theta_k^{(n)}+\frac{\pi}{2n}} |t-\theta_k^{(n)}| \, dt \, |S_{k,n}(\theta)|+\\
		&\quad \frac{2n}{\pi}\|g'\|_\infty \sum_{k=1}^n\int\limits_{\theta_k^{(n)}}^{\theta_k^{(n)}+\frac{\pi}{2n}} |\theta-\theta_k^{(n)}| \, dt \, |S_{k,n}(\theta)|\\
		&\leq \frac{2n}{\pi}\|g'\|_\infty\left (\frac{\pi}{2n}\right )^2\sum_{k=1}^n|S_{k,n}(\theta)|+\frac{2n}{\pi}\|g'\|_\infty\left (\frac{\pi}{2n}\right )\sum\limits_{k=1}^n|\theta-\theta_k^{(n)}||S_{k,n}(\theta)|\\
		&\leq c_1\|g'\|_\infty\left (\frac{\pi}{2n}\right )+\|g'\|_\infty\sum\limits_{k=1}^n|\theta-\theta_k^{(n)}||S_{k,n}(\theta)|.
		\end{align*}
		Applying Minkowski's inequality, we have 
	    \[
	    \|\mathcal{GK}_n(g)-g\|_p\leq c_1\|g'\|_\infty\left (\frac{\pi}{2n}\right )+\|g'\|_\infty\sum\limits_{k=1}^n\left (\int\limits_{0}^{\pi} |\theta-\theta_k^{(n)}|^p \, |S_{k,n}(\theta)|^p \, d\theta\right )^\frac{1}{p}.
	    \]
	     \textbf{Case 1: $p=1$} By Proposition \ref{propos}, we get 
	    \[
	    \|\mathcal{GK}_n(g)-g\|_1\leq \tilde C_{1}\|g'\|_\infty\left (\frac{1+\log n}{n}\right ).
	    \]
	    \textbf{Case 2: $1<p<+\infty$} Again applying Proposition \ref{propos}, we obtain
	    \[
	    \|\mathcal{GK}_n(g)-g\|_p\leq \tilde C_{p}\|g'\|_\infty\left (\frac{1}{n}+\frac{1}{n^\frac{1}{p}}\right ),
	    \]
	    where $C_{1,p}>0$ and $C_{2,p}>0$ are constants. 
	    
	    Let 
	    \[
	    m_{n}=\begin{cases}
	    	\frac{1+\log n}{n}, & p = 1, \\[8pt]
	    	\frac{1}{n}+\frac{1}{n^\frac{1}{p}}, & p > 1.
	    \end{cases}
	    \]
	    Now, consider
	    \begin{align*}
	    	\|\mathcal{GK}_n(f)-f\|_p&
	    	\leq (C_p+1)\|f-g\|_p+\|\mathcal{GK}_n(g)-g\|_p\\
	    	&\leq (C_p+1)\|f-g\|_p+\tilde C_pm_{n}\|g'\|_\infty\\
	    	&=(C_p+1)\{\|f-g\|_1+\frac{\tilde C_p}{C_p+1}m_{n}\|g'\|_\infty\}.
	    \end{align*}
	    Now taking infimum over $g\in C^1[0,\pi]$, we have
	    \[ 
	    \|\mathcal{GK}_n(f)-f\|_p\leq \tilde C_pK_p(f,\tilde R_pm_{n}),
	    \]
	    where $R_p=C_p+1$ and $\tilde R_p=\frac{\tilde C_p}{(C_p+1)}$.
	\end{proof}

	\section{Point-wise Boundedness of $\{\mathcal{GK}_n\}_{n\in\mathbb{N}}$ using the Hardy-Littlewood Maximal Operator}\label{section4}
	In this section, we establish the point-wise boundedness of the sequence $\{\mathcal{GK}_n\}_{n\in\mathbb{N}}$ using the Hardy-Littlewood maximal operator.
	The Hardy--Littlewood maximal operator $M$ is defined for $f \in L^1[a,b]$ by
	\[
	Mf(x) := \sup_{I \ni x} \frac{1}{|I|} \int_I |f(t)|\, dt,
	\]
	where the supremum is taken over all intervals $I \subset [a,b]$ containing $x$. Thus, $Mf(x)$ is defined for all $x \in [a,b]$.
	We do not obtain the point-wise bound for $\{\mathcal{GK}_n\}_{n\in\mathbb{N}}$ on the entire interval $[0,\pi]$; nevertheless, it holds on $[\varepsilon,\pi-\varepsilon]$ for all $0<\varepsilon<\pi$, as shown below.
	\begin{theorem}\label{Mains}
		Let $M$ be the Hardy-Littlewood maximal operator defined above. Let $0<\epsilon<\pi$, $\theta\in [\epsilon,\pi-\epsilon]$ and $f\in L^1[0,\pi]$. Then for all $n=1,2,\ldots$,
		\[
		|\mathcal{GK}_n(f)(\theta)|\leq C_\epsilon Mf(\theta),
		\]
		where $C_\epsilon>0$ is a constant which does not depend on $n$, $f$ or $\theta$.
	\end{theorem}
	\begin{proof}
		Let $\theta\in [0,\pi]$. Then  
		\begin{equation}\label{eqq}
			\mathcal{GK}_n(f)(\theta)=\frac{2n}{\pi}\sum\limits_{k=1}^n \int\limits_{\theta_k^{(n)}}^{\theta_k^{(n)}+\frac{\pi}{2n}}f(t)\ dt S_{k,n}(\theta).
		\end{equation}
		Suppose $\theta\in [\theta_j^{(n)},\theta_{j}^{(n)}+\frac{\pi}{2n}]$ for some $j\in\{1,\ldots, n\}$. Then we split the sum as follows.
		\begin{equation*}
			\mathcal{GK}_n(f)(\theta)=\frac{2n}{\pi}
			\sum\limits_{\substack{1 \le k \le n \\ k \neq j}} \int\limits_{\theta_k^{(n)}}^{\theta_k^{(n)}+\frac{\pi}{2n}}f(t)\ dt S_{k,n}(\theta)+\frac{2n}{\pi}
			\int\limits_{\theta_j^{(n)}}^{\theta_j^{(n)}+\frac{\pi}{2n}}f(t)\ dt S_{j,n}(\theta).
		\end{equation*}
		If $\theta\in [0,\theta_1^{(n)}]$, we have (\ref{eqq}).
		Consider the first sum.  
		Taking modulus, we have 
		\begin{align*}
			\frac{2n}{\pi}
			\left| \sum_{\substack{1 \le k \le n \\ k \neq j}}
			\int\limits_{\theta_k^{(n)}}^{\theta_k^{(n)}+\frac{\pi}{2n}} f(t)\, dt \, S_{k,n}(\theta) \right|
			&\leq \frac{2n}{\pi}
			\sum_{k=1}^{j-1} \int\limits_{\theta_k^{(n)}}^{\theta_k^{(n)}+\frac{\pi}{2n}} |f(t)|\, dt \, |S_{k,n}(\theta)| \\
			&\quad + \frac{2n}{\pi}
			\sum_{k=j+1}^{n} \int\limits_{\theta_k^{(n)}}^{\theta_k^{(n)}+\frac{\pi}{2n}} |f(t)|\, dt \, |S_{k,n}(\theta)| \\
			&\leq \frac{2n}{\pi}
			\sum\limits_{k=1}^{j-1} \int\limits_{\theta_k^{(n)}}^{\theta} |f(t)|\, dt \, |S_{k,n}(\theta)| \\
			&\quad + \frac{2n}{\pi}
			\sum_{k=j+1}^{n} \int\limits_{\theta}^{\theta_k^{(n)}+\frac{\pi}{2n}} |f(t)|\, dt \, |S_{k,n}(\theta)| \\
			&= \frac{2n}{\pi}
			\sum_{k=1}^{j-1} \frac{1}{|\theta-\theta_k^{(n)}|}
			\int\limits_{\theta_k^{(n)}}^{\theta} |f(t)|\, dt \,
			|\theta-\theta_k^{(n)}| \, |S_{k,n}(\theta)| \\
			&\quad + \frac{2n}{\pi}
			\sum_{k=j+1}^{n} \frac{1}{\left|\theta-\theta_k^{(n)}-\frac{\pi}{2n}\right|}
			\int\limits_{\theta}^{\theta_k^{(n)}+\frac{\pi}{2n}} |f(t)|\, dt \\
			&\qquad \times \left|\theta-\theta_k^{(n)}-\frac{\pi}{2n}\right| \, |S_{k,n}(\theta)|.
		\end{align*}
		By the definition of the Hardy Littlewood maximal operator, we have 
		\[
		\frac{1}{|\theta-\theta_k^{(n)}|}
		\int_{\theta_k^{(n)}}^{\theta} |f(t)|\, dt
		\leq Mf(\theta), 
		\quad 
		\frac{1}{\left|\theta-\theta_k^{(n)}-\frac{\pi}{2n}\right|}
		\int_{\theta}^{\theta_k^{(n)}+\frac{\pi}{2n}} |f(t)|\, dt
		\leq Mf(\theta).
		\]
		
		Thus, we have
		\begin{align*}
			\frac{2n}{\pi}
			\left| \sum_{\substack{1 \le k \le n \\ k \neq j}} 
			\int_{\theta_k^{(n)}}^{\theta_k^{(n)}+\frac{\pi}{2n}} f(t)\, dt \, S_{k,n}(\theta) \right|
			&\leq Mf(\theta)\frac{2n}{\pi}
			\sum_{k=1}^{j-1} |\theta-\theta_k^{(n)}| \, |S_{k,n}(\theta)| \\
			&\quad + Mf(\theta)\frac{2n}{\pi}
			\sum_{k=j+1}^{n} \left|\theta-\theta_k^{(n)}-\frac{\pi}{2n}\right| \, |S_{k,n}(\theta)| \\
			&\leq Mf(\theta)\frac{2n}{\pi}
			\sum_{k=1}^{n} |\theta-\theta_k^{(n)}| \, |S_{k,n}(\theta)| \\
			&\quad + Mf(\theta)\frac{2n}{\pi}\frac{\pi}{2n}
			\sum_{k=j+1}^{n} |S_{k,n}(\theta)| \\
			&\leq Mf(\theta)\frac{2n}{\pi}
			\sum_{k=1}^{n} |\theta-\theta_k^{(n)}| \, |S_{k,n}(\theta)|
			+ c_1\, Mf(\theta).
		\end{align*}
		Therefore, we deduce
		\begin{equation*}
			|\mathcal{GK}_n(f)(\theta)|\leq Mf(\theta)\frac{2n}{\pi}\sum\limits_{k=1}^{n}|\theta-\theta_k^{(n)}||S_{k,n}(\theta)|+2c_1Mf(\theta).
		\end{equation*}
		Consider 
		\[
		\sum_{k=1}^{n} |\theta-\theta_k^{(n)}| \, |S_{k,n}(\theta)|
		= \sum_{k=1}^{n} \left| \arccos(\cos\theta) - \arccos(\cos\theta_k^{(n)}) \right| \, |S_{k,n}(\theta)|.
		\]
		
		From \cite{mills1}, for $f \in C[-1,1]$ and $x \in [-1,1]$ with $\cos\theta = x$, we have the inequality
		\[
		\sum_{k=1}^{n} \left| f(\cos\theta) - f(\cos\theta_k^{(n)}) \right| \, |S_{k,n}(\theta)|
		= \mathcal{O}\!\left( \omega\!\left( \frac{\sqrt{1-x^2}}{n} \right) + \omega\!\left( \frac{1}{n^2} \right) \right)
		= \mathcal{O}\!\left( \omega\!\left( \frac{1}{n} \right) \right).
		\]
		Here, we take $f(x) = \arccos x$ for $x \in [-1,1]$. Let $\pi > \varepsilon > 0$ and suppose that $\theta \in [\varepsilon,\pi-\varepsilon]$. Then, it holds that
		\[
		\left| \arccos(\cos\theta) - \arccos(\cos\theta_k^{(n)}) \right|
		\leq \frac{1}{\sin\!\left( \frac{\theta+\theta_k^{(n)}}{2} \right)}
		|\theta-\theta_k^{(n)}|
		\leq \frac{1}{\sin\!\left( \frac{\varepsilon}{2} \right)}
		|\theta-\theta_k^{(n)}|.
		\]
		
		Consequently, $\omega\!\left( \arccos, \frac{1}{n} \right)
		\leq \frac{1}{\sin\!\left( \frac{\varepsilon}{2} \right)} \cdot \frac{1}{n}.$ Therefore, we have
		\[
		\sum_{k=1}^{n} |\theta-\theta_k^{(n)}| \, |S_{k,n}(\theta)|
		= \frac{1}{\sin\!\left( \frac{\varepsilon}{2} \right)} \, \mathcal{O}\!\left( \frac{1}{n} \right).
		\]
		
		Hence, we get
		\[
		|\mathcal{GK}_n(f)(\theta)|
		\leq C \, Mf(\theta)\, \frac{1}{\sin\!\left( \frac{\varepsilon}{2} \right)}
		+ 2c_1\, Mf(\theta)
		= C_\varepsilon \, Mf(\theta),
		\]
		where $C > 0$ is a constant and $C_\varepsilon = C \frac{1}{\sin\!\left( \frac{\varepsilon}{2} \right)} + 2c_1.$
		
	\end{proof}
	\section{Korovkin-type results in Banach Function spaces}\label{section5}
	
	Korovkin-type theorems play a fundamental role in approximation theory, providing powerful criteria to verify the convergence of sequences of positive linear operators by testing them on a small set of functions \cite{korovkin, altomare}. This remarkable result unified several classical approximation theorems in the literature such as due to Bernstein, Weierstrass, Fe\'jer etc. These results have been extended to various settings, including that of Banach function spaces. 
	
	In this section, we recall some fundamental definitions concerning Banach function spaces and state a Korovkin-type theorem in this setting. Using this theorem, we extend the convergence results for $\{\mathcal{GK}_n\}_{n\in\mathbb{N}}$ obtained in $L^p[0,\pi]$ to a more general class of Banach function spaces.
	
	Below, we recall the definition of a function norm. 
	
	Let $(A, \mathcal{S}, \mu)$ be a measurable space, where $\mathcal{S}$ is a $\sigma$-algebra of measurable subsets of $A$. Let $\mathcal{M}$ denote the set of measurable functions on $A$, and let $\mathcal{M}^+$ denote the set of non-negative measurable functions on $A$.
	
	\begin{definition}
		A mapping $\rho: \mathcal{M}^+ \to [0, +\infty]$ is called a function norm if the following properties hold for all $f, g, f_n \in \mathcal{M}^+$, $a \geq 0$, and $E \in \mathcal{S}$:
		\begin{enumerate}
			\item $\rho(f) = 0$ if and only if $f = 0$ $\mu$-a.e., $\rho(af) = a \rho(f)$, and $\rho(f+g) \leq \rho(f) + \rho(g)$;
			\item If $g \leq f$ $\mu$-a.e., then $\rho(g) \leq \rho(f)$;
			\item If $f_n \uparrow f$ $\mu$-a.e., then $\rho(f_n) \uparrow \rho(f)$;
			\item If $\mu(E) < +\infty$, then $\rho(\chi_E) < +\infty$;
			\item If $\mu(E) < +\infty$, then there exists a constant $C_E > 0$ such that
			\[
			\int\limits_E f \, d\mu \leq C_E \rho(f),
			\]
			where $C_E$ depends on $E$ and $\rho$, but not on $f$.
		\end{enumerate}
	\end{definition}
	
	A Banach function space $X$ generated by $\rho$ is a Banach space of functions $f \in \mathcal{M}$ equipped with the norm $\|f\|_X := \rho(|f|)$.
	
	\begin{definition}
		A function $f \in X$ is said to have an absolutely continuous norm if
		\[
		\|f \chi_{E_n}\|_X \to 0 \quad \text{as } n \to \infty,
		\]
		for every sequence $\{E_n\} \subset \mathcal{S}$ such that $E_n \to \emptyset$ $\mu$-a.e. (i.e., $\chi_{E_n} \to 0$ $\mu$-a.e.).
	\end{definition}
	
	Define
	\[
	X_a = \left \{ f \in X : f \text{ has an absolutely continuous norm} \right \}.
	\]
	
	\textbf{Special case:} $A = [a,b]$, $\mathcal{S}$ is the Borel $\sigma$-algebra, and $\mu$ is the Lebesgue measure.
	
	We now discuss a special case of Banach function spaces. Without loss of generality, we assume $A = [0,1]$. Let $\mu$ be the Lebesgue measure and $\mathcal{S}$ the Borel $\sigma$-algebra. Let $\rho$ be a function norm on this measure space, and let $X$ be the corresponding Banach function space with norm $\|f\|_X = \rho(|f|)$.
	
	For $\delta > 0$, define the shift operator $T_\delta$ on $X$ by
	\[
	T_\delta(f)(x) =
	\begin{cases}
		f(x+\delta), & \text{if } x+\delta \in [0,1], \\
		0, & \text{otherwise}.
	\end{cases}
	\]
	
	Let $X^S$ denote the subspace of $X$ defined by the closure in $X$ of the set
	\[
	\left \{ f \in X : \lim_{\delta \to 0} \|T_\delta(f) - f\|_X = 0 \right \}.
	\]
	It is known that $X^S$ contains $C[0,1]$ and is, in fact, strictly larger than $C[0,1]$ (see \cite{vinaya, zeren}). 
	
	In 2022, Zeren et al.~\cite{zeren} established a Korovkin-type theorem on the subspace $X^S$. Subsequently, in 2025, Kiran Kumar and Vinaya extended this result to a non-positive version, stated as follows. 
	Let $B(X^S)$ denote the space of all bounded linear operators on $X^S$.
	\begin{theorem}\cite{vinaya}\label{nonpositive}
		Let $X$ be a Banach function space such that $1 \in X_a$, and let $\{L_n\}_{n \in \mathbb{N}}$ be a sequence of bounded linear operators on $X^S$ satisfying:
		\begin{enumerate}
			\item $\lim\limits_{n \to \infty} L_n(g) = g$ in $C[0,1]$ for all $g \in \{1, t, t^2\}$;
			\item $\sup\limits_{n} \|L_n\|_{B(C[0,1])} < \infty$.
		\end{enumerate}
		Then $\lim\limits_{n \to \infty} L_n(f) = f$ in $X$ for all $f \in X^S$ if and only if $\sup\limits_{n} \|L_n\|_{B(X^S)} < \infty.$
	\end{theorem}
	As already mentioned the operators $\mathcal{GK}_n$, for $n=1,2,\ldots$ are non-positive. Therefore, we employ this non-positive Koorvkin-type theorem, to extend the convergence results to general Banach function spaces.
	
	Clearly, the sequence of operators $\{\mathcal{GK}_n\}_{n\in\mathbb{N}}$ satisfies conditions (1) and (2) on $C[0,\pi]$ (see Section \ref{section1}). To apply Theorem \ref{nonpositive}, it therefore suffices to verify that $\sup\limits_{n} \|\mathcal{GK}_n\|_{B(X^S)} < \infty.$
	In the following section, we establish this boundedness condition and consequently apply Theorem \ref{nonpositive} to several explicit examples of Banach function spaces.
	\section{Boundedness and Convergence of $\{\mathcal{GK}_n\}_{n\in\mathbb{N}}$ in some Banach Function Spaces}
	We observe that the sequence of operators $\{\mathcal{GK}_n\}_{n\in\mathbb{N}}$ is non-positive due to the non-positivity of the Lagrange interpolation polynomials. Our aim is to extend the convergence of the sequence $\{\mathcal{GK}_n\}$ to the subspace $X^S$ for certain Banach function spaces. To this end, we apply Theorem \ref{nonpositive}.
	
	Conditions (1) and (2) have already been established in Section \ref{section1}. Therefore, it remains to verify the uniform boundedness of this sequence of operators on $X$, a given Banach function space. To achieve this, we make use of the point-wise estimate involving the Hardy-Littlewood maximal operator obtained in Section \ref{section4}.
	
	We now prove the boundedness, and hence the convergence, of the sequence $\{\mathcal{GK}_n\}_{n\in\mathbb{N}}$ on several explicit Banach function spaces presented below.
	\begin{example}\textbf{Weighted Lebesgue space}
		
		Let \( X = L^{p,w}(0,\pi) \), \( 1 < p < +\infty \), be a weighted Lebesgue space of measurable functions \( f \) on \([0,\pi]\) with the norm
		\[
		\|f\|_{p,w}
		=
		\left(
		\int\limits_0^\pi |f(t)|^p w(t)\,dt
		\right)^{\frac{1}{p}},
		\]
		where the weight function \( w \) satisfies the Muckenhoupt condition. Then the Hardy-Littlewood maximal operator $M$ is bounded on \( L^{p,w}(0,\pi) \) (see \cite{zeren}). On restricting the interval to $[\epsilon,\pi-\epsilon]$ for $0<\epsilon<\pi$ and applying Theorem \ref{Mains}, for $f\in L^{p.w}(0,\pi)$, we have 
		\[
		\|\mathcal{GK}_n\|_{p,w,\epsilon}\leq \|Mf\|_{p,w,\epsilon}\leq C_\epsilon \|f\|_{p,w,\epsilon}, 
		\]
		where $\|.\|_{p,w,\epsilon}$ denotes the weighted $L^p$ norm on $[\epsilon,\pi-\epsilon]$. Applying Theorem \ref{nonpositive}, we already have $\lim\limits_{n\to\infty}\|\mathcal{GK}_n(f)-f\|_{p,w,\epsilon}=0$ for all $f\in (L^{p.w}[0,\pi])^S$.
	\end{example}
	\begin{example}\textbf{Grand Lebesgue Space}
		
		Let \( X = L^{p}_G(0,\pi) \), \( 1 < p < +\infty \), be a grand Lebesgue space consisting of measurable functions \( f \) on \([0,\pi]\) with the norm
		\[
		\|f\|_G^{p}
		=
		\sup_{0<h<p-1}
		h^{\frac{1}{p-h}} \, \|f\|_{p-h}.
		\]
		
		Then the subspace \( (L^{p}_G(0,\pi))^S \) consists of functions \( f \) satisfying the condition
		\[
		\lim_{h \to +0}
		h \int\limits_0^\pi |f(t)|^{p-h}\,dt = 0.
		\]
		Moreover the maximal function is bounded in this space (see \cite{zeren}). Thus, on restricting to the interval $[\epsilon, \pi-\epsilon]$ and applying Theorem \ref{Mains} and Theorem \ref{nonpositive}, we have 
		\[
		\|\mathcal{GK}_n(f)\|_{G}^{p,\epsilon}\leq \|Mf\|_{G}^{p,\epsilon}\leq C_\epsilon\|f\|_{G}^{p,\epsilon}
		\]
		and 
		\[
		\lim\limits_{n \to \infty} \|\mathcal{GK}_n f-f\|_{G}^{p,\epsilon} \quad \text { for all } f \in (L^{p}_G(0,\pi))^S,
		\]
		where $\|.\|_G^{p,\epsilon}$ denotes the norm $\|.\|_G^{p}$ restricted to the interval $[\epsilon,\pi-\epsilon]$.
	\end{example}
	
	\begin{example}\textbf{Weighted grand Lebesgue space}
		
		Let \( X = L^{p,w}_G(0,\pi) \), \( 1 < p < +\infty \), be a weighted grand Lebesgue space consisting of measurable functions \( f \) on \([0,\pi]\) with the norm
		\[
		\|f\|_G^{p,w}
		=
		\sup_{0<h<p-1}
		h
		\left(
		\int\limits_0^1 |f(t)|^{p-h} w(t)\,dt
		\right)^{\frac{1}{p-h}},
		\]
		where the weight \( w \) satisfies the Muckenhoupt condition. By the boundedness of the Hardy-Littlewood maximal operator \( M \) in this space (see \cite{zeren}), applying Theorem \ref{Mains} and Theorem \ref{nonpositive} we have 
		\[
		\|\mathcal{GK}_n(f)\|_{G}^{p,w,\epsilon}\leq \|Mf\|_{G}^{p,w,\epsilon}\leq C_\epsilon\|f\|_{G}^{p,w,\epsilon}
		\]
		and 
		\[
		\lim\limits_{n \to \infty} \|\mathcal{GK}_n f-f\|_{G}^{p,w,\epsilon} \quad \text { for all } f \in (L^{p}_G(0,\pi))^S,
		\]
		where $\|.\|_G^{p,w,\epsilon}$ denotes the norm $\|.\|_G^{p,w}$ restricted to the interval $[\epsilon,\pi-\epsilon]$.
	\end{example}
	\begin{example}\textbf{Orlicz space}
		
		Let \( \varphi(t) \ge 0 \), $\psi(t) = \sup_{\varphi(s)\le t} s, \ t \ge 0,$ $\Phi(t) = \int_0^t \varphi(s)\,ds,$ and  $\Psi(t) = \int_0^t \psi(s)\,ds.$
		
		Let \( X = L^\Phi(0,\pi) \) be an Orlicz space with Young function \( \Phi \), consisting of measurable functions \( f \) on \([0,\pi]\) with the norm
		\[
		\|f\|_{\Phi}
		=
		\sup_{g \in S^\Psi}
		\int\limits_0^\pi |f(x)g(x)|\,dx,
		\]
		where
		$S^\Psi = \left\{ g \in L^1(0,\pi) : \int_0^\pi \Psi(|g(x)|)\,dx \le 1 \right\}.$ Assuming that \( \Phi \) satisfies the \( \Delta_2 \)-condition, i.e., there exist constants \( k>0 \), \( T \ge 0 \) such that
		$\Phi(2t) \ge k \Phi(t), \ t \ge T$, the Hardy-Littlewood maximal operator \( M \) is bounded in \( L^\Phi(0,\pi) \) (see \cite{zeren,orlicz}). Applying Theorem \ref{Mains} and Theorem \ref{nonpositive}, we have 
		\[
		\|\mathcal{GK}_n(f)\|_{\Phi,\epsilon}\leq \|Mf\|_{\Phi,\epsilon}\leq C_\epsilon\|f\|_{\Phi,\epsilon}
		\]
		and 
		\[
		\lim\limits_{n\to\infty}\|\mathcal{GK}_n(f)-f\|_{\Phi,\epsilon}=0\ \quad \text{for all } f\in L^\Phi(0,\pi)^S, 
		\]
		where $\|.\|_{\Phi,\epsilon}$ denotes the norm $\|.\|_\Phi$ restricted to the interval $[\epsilon,\pi-\epsilon]$.
	\end{example}
	\begin{example}\textbf{Lebesgue space with variable exponent}
		
		Let \( X = L^{p(\cdot)}(0,\pi) \), \( 1 < p(x) < +\infty \), be the variable exponent Lebesgue space with norm
		\[
		\|f\|_{p(\cdot)}
		=
		\inf\left\{
		\lambda > 0 :
		\int\limits_0^\pi \left|\frac{f(t)}{\lambda}\right|^{p(t)} dt \le 1
		\right\}.
		\]
		
		Then  \( X^S = L^{p(\cdot)}(0,\pi) \) \cite{variableexpo}. Assume that
		\[
		|p(x) - p(y)| \le \frac{c_0}{-\log |x-y|}, \quad |x-y| < \tfrac{1}{2},
		\]
		and \( 1 < p^- \le p^+ < +\infty \), where
		$p^- = \operatorname*{ess\,inf}\limits_{x \in [0,\pi]} p(x), \
		p^+ = \operatorname*{ess\,sup}\limits_{x \in [0,\pi]} p(x).$ Then the maximal operator \( M \) is bounded in \( L^{p(\cdot)}(0,\pi) \). Hence by Theorem \ref{Mains} and Theorem \ref{nonpositive} we have 
		\[
		\|\mathcal{GK}_n(f)\|_{p(.),\epsilon}\leq \|Mf\|_{p(.),\epsilon}\leq C_\epsilon\|f\|_{p(.),\epsilon}
		\]
		and 
		\[
		\lim\limits_{n\to\infty}\|\mathcal{GK}_n(f)-f\|_{p(.),\epsilon}=0\ \quad \text{for all } f\in L^{p(.)}(0,\pi), 
		\]
		where $\|.\|_{p(.),\epsilon}$ denotes the norm $\|.\|_{p(.)}$ restricted to the interval $[\epsilon,\pi-\epsilon]$.
	\end{example}
	\begin{example}\textbf{Morrey space}
		
		Let \( 1 < p \le p_0 < +\infty \) and \( X = M^{p_0}_p(0,\pi) \) be the Morrey space with norm
		\[
		\|f\|_{M^{p_0}_p(0,\pi)}
		=
		\sup_{I \subset [0,\pi]}
		|I|^{\frac{1}{p_0}}
		\left(
		\frac{1}{|I|} \int\limits_I |f(t)|^p dt
		\right)^{\frac{1}{p}},
		\]
		where the supremum is taken over all intervals \( I \subset [0,\pi] \). Then \( M \) is bounded in \( M^{p_0}_p(0,\pi) \) (see \cite{morrey, morrey1, zeren}). Hence applying Theorem \ref{Mains} and Theorem \ref{nonpositive} we have 
		\[
		\|\mathcal{GK}_n(f)\|_{M^{p_0}_p(\epsilon,\pi-\epsilon)}\leq \|Mf\|_{M^{p_0}_p(\epsilon,\pi-\epsilon)}\leq C_\epsilon\|f\|_{M^{p_0}_p(\epsilon,\pi-\epsilon)}
		\]
		and 
		\[
		\lim\limits_{n\to\infty}\|\mathcal{GK}_n(f)-f\|_{M^{p_0}_p(\epsilon,\pi-\epsilon)}=0\ \quad \text{for all } f\in {M^{p_0}_p(0,\pi)}.
		\]
	\end{example}
	\begin{example}\textbf{Weighted Morrey space}
		
		Let \( 1 < p \le p_0 < +\infty \), \( w \) be a weight, and \( X = M^{p_0}_{p,w}(0,1) \) with norm
		\[
		\|f\|_{M^{p_0}_{p,w}(0,1)}
		=
		\sup_{I \subset [0,1]}
		|I|^{\frac{1}{p_0}}
		\left(
		\frac{1}{|I|} \int\limits_I |f(t)|^p w(t)\,dt
		\right)^{\frac{1}{p}}.
		\]
		
		Assume that
		\[
		\sup_{I' \subset I}
		\left(
		\frac{|I|}{|I'|}
		\right)^{\frac{p}{p_0}}
		\left(
		\frac{1}{|I|} \int\limits_I w(t)\,dt
		\right)
		\left(
		\frac{1}{|I'|} \int\limits_{I'} w(t)^{-\frac{1}{p-1}} dt
		\right)^{p-1}
		< \infty.
		\]
		
		Then the maximal operator \( M \) is bounded in \( M^{p_0}_{p,w}(0,1) \) (see \cite{morrey1, zeren}), and hence using 
		Theorem \ref{Mains} and Theorem \ref{nonpositive} we have 
		\[
		\|\mathcal{GK}_n(f)\|_{M^{p_0}_{p,w}(\epsilon,\pi-\epsilon)}\leq \|Mf\|_{M^{p_0}_{p,w}(\epsilon,\pi-\epsilon)}\leq C_\epsilon\|f\|_{M^{p_0}_{p,w}(\epsilon,\pi-\epsilon)}
		\]
		and 
		\[
		\lim\limits_{n\to\infty}\|\mathcal{GK}_n(f)-f\|_{M^{p_0}_{p,w}(\epsilon,\pi-\epsilon)}=0\ \quad \text{for all } f\in {M^{p_0}_{p,w}(0,\pi)}.
		\]
	\end{example}
	\begin{example}\textbf{Morrey Space $\mathcal{M}^{p,\lambda}(0,\pi)$}
		
		Let $X = \mathcal{M}^{p,\lambda}(0,\pi)$, where $1 \le p < \infty$ and $0 < \lambda < 1$, be the Morrey space of measurable functions $f$ on $[0,\pi]$ with the norm
		\[
		\|f\|_{\mathcal{M}^{p,\lambda}}
		=
		\sup_{E \subset [0,\pi]}
		\left(
		\frac{1}{|E|^\lambda}
		\int\limits_E |f(t)|^p \, dt
		\right)^{\frac{1}{p}},
		\]
		where the supremum is taken over all measurable sets $E \subset [0,\pi]$. Then \( M \) is bounded in \( M^{p,\lambda}(0,\pi) \) as shown in \cite{zeren}. Using 
		Theorem \ref{Mains} and Theorem \ref{nonpositive} we have 
		\[
		\|\mathcal{GK}_n(f)\|_{M^{p,\lambda}(\epsilon,\pi-\epsilon)}\leq \|Mf\|_{M^{p,\lambda}(\epsilon,\pi-\epsilon)}\leq C_\epsilon\|f\|_{M^{p,\lambda}(\epsilon,\pi-\epsilon)}
		\]
		and 
		\[
		\lim\limits_{n\to\infty}\|\mathcal{GK}_n(f)-f\|_{M^{p,\lambda}(\epsilon,\pi-\epsilon)}=0\ \quad \text{for all } f\in {M^{p,\lambda}(0,\pi)}.
		\]
	\end{example}
	
	\section*{Concluding Remarks}
	In this paper, we introduced a new sequence of operators based on the Gr\"unwald interpolation operators on $C[0,\pi]$ and $L^p[0,\pi]$, $1 \leq p < +\infty$. We established the uniform boundedness of this sequence in $C[0,\pi]$ and $L^1[0,\pi]$, and consequently extended it to $L^p[0,\pi]$ for $1 < p < +\infty$ using the Riesz--Thorin interpolation theorem. We then proved the convergence of this sequence and obtained the corresponding rates of convergence using the modulus of continuity and a suitable $K$-functional. Furthermore, we derived a point-wise estimate for this sequence via the Hardy-Littlewood maximal operator. This allowed us to extend the convergence results to several Banach function spaces on a nontrivial subspace by applying a Korovkin-type theorem from \cite{vinaya}. 
	
	The extension of these convergence results to the entire Banach function space remains an open problem. Moreover, obtaining quantitative estimates of convergence in this setting is a direction for future research. The Gr\"unwald--Kantorovich operators for other choice of nodes and their convergence properties will be discussed elsewhere. 
	\section*{Acknowledgments}
	The author wishes to thank Dr. V. B. Kiran Kumar and Prof. M. N. N. Namboodiri for fruitful discussions on Gr\"unwald interpolation operators and related topics. She also thanks the organizers of the workshop "Commutative and Non-Commutative Harmonic Analysis", held at NISER Bhubaneshwar, India during March 2-7, 2026, where she gained valuable insights into several important topics in harmonic analysis.

\end{document}